\documentclass[10pt, reqno]{amsart}

\usepackage{amsmath,amssymb,amsthm, epsfig}
\usepackage{mathalfa}
\usepackage{amsfonts}
\usepackage{amsmath}
\usepackage{amssymb}
\usepackage[mathscr]{euscript}
\usepackage{esint}

\def \N {\mathbb{N}}
\def \R {\mathbb{R}}

\newtheorem{theorem}{Theorem}
\newtheorem{lemma}[theorem]{Lemma}
\newtheorem{proposition}[theorem]{Proposition}

\theoremstyle{definition}

\theoremstyle{remark}

\numberwithin{equation}{section}

\newcommand{\intav}[1]{\mathchoice {\mathop{\vrule width 6pt height 3 pt depth  -2.5pt
\kern -8pt \intop}\nolimits_{\kern -6pt#1}} {\mathop{\vrule width
5pt height 3  pt depth -2.6pt \kern -6pt \intop}\nolimits_{#1}}
{\mathop{\vrule width 5pt height 3 pt depth -2.6pt \kern -6pt
\intop}\nolimits_{#1}} {\mathop{\vrule width 5pt height 3 pt depth
-2.6pt \kern -6pt \intop}\nolimits_{#1}}}

\title[Sharp regularity for the inhomogeneous pme]{Sharp regularity for the inhomogeneous\\ porous medium equation}

\author[D.J.~Ara\'ujo]{Dami\~ao J. Ara\'ujo}
\address{Department of Mathematics, Universidade Federal da Para\'iba, Campus Universit\'a\-rio, Jo\~ao Pessoa - PB - Brazil, 58059-900}{}
\email{araujo@mat.ufpb.br}

\author[A. F. Maia]{Anderson F. Maia}
\address{CMUC, Department of Mathematics, University of Coimbra, 3001-501 Coimbra, Portugal}{}
\email{anderson.maia@student.uc.pt}

\author[J.M.~Urbano]{Jos\'{e} Miguel Urbano}
\address{CMUC, Department of Mathematics, University of Coimbra, 3001-501 Coimbra, Portugal}{}
\email{jmurb@mat.uc.pt}

\begin{document}



\begin{abstract} 
We show that locally bounded solutions of the inhomogeneous porous medium equation
$$u_{t} - {\rm div} \left( m |u|^{m-1} \nabla u \right) = f \in L^{q,r}, \quad m >1 ,$$
are locally H\"older continuous, with exponent
$$\gamma =\min \left\{ \frac{\alpha_{0}^-}{m}, \frac{[(2q - n)r -2q]}{q[(mr - (m-1)]} \right\},$$
where $\alpha_{0}$ denotes the optimal H\"older exponent for solutions of the homogeneous case.
The proof relies on an approximation lemma and geometric iteration in the appropriate intrinsic scaling.
\end{abstract}

\date{\today}

\maketitle

\section{Introduction} \label{sct intro}

The quest for obtaining sharp, optimal regularity results is one of the most exciting current trends in the study of nonlinear pdes. Degenerate parabolic equations are known to have H\"older continuous solutions (cf. \cite{DiBe93, Urb08}) under quite general structure assumptions, corresponding to the archetypal $p$-Laplace equation and porous medium equation (pme). The main difference between these two extensively studied pdes is that the first degenerates at points where the gradient of a solution vanishes and the second at points where this happens for the solution itself. The regularity theory for both equations has evolved in parallel and results for one normally have a counterpart for the other. Recently, in \cite{TeiUrb14}, the sharp H\"older exponent
$$ \frac{(pq-n)r-pq}{q[(p-1)r-(p-2)]} $$
for weak solutions of the inhomogeneous $p$-Laplace equation was determined precisely only in terms of $p$, the space dimension $n$ and the $L^{q,r}$-integrability of the source. Inspired by the recent breakthroughs in \cite{AraTeiUrb18, AraTeiUrb17, AraZhang17}, our goal in this paper is to do the same for the porous medium equation (cf. \cite{Vaz07}). 

Let $U \subset \R^n$ be open and bounded, $T>0$ and $U_T=U\times (0,T)$. We consider the prototype inhomogeneous equation
\begin{equation}
u_{t} - {\rm div} \left( m |u|^{m-1} \nabla u \right) = f, \quad m >1 ,
\label{pme}
\end{equation}
with a source term $f \in L^{q,r} (U_T) \equiv L^r(0,T;L^q(U))$, where
\begin{equation} \label{borderline1}
\frac{1}{r}+\frac{n}{2q} <1,
\end{equation}
which is the standard minimal integrability condition that guarantees the existence of bounded weak solutions and their H\"older regularity. 

We will show that  bounded weak solutions of \eqref{pme} are locally of class $C^{0,\gamma}$ in space, with
$$\gamma =\frac{\alpha}{m}  , \qquad \alpha =  \min\left\{\alpha_{0}^-, \frac{m[(2q - n)r -2q]}{q[mr - (m-1)]}\right\} ,$$
where $0< \alpha_{0} \leq 1$ denotes the optimal H\"older exponent for solutions of \eqref{pme} with $f \equiv 0$. The regularity class is to be interpreted in the following sense: if 
$$\frac{m[(2q - n)r -2q]}{q[mr - (m-1)]} < \alpha_0$$
then solutions are in $C^{0,\gamma}$, with 
$$\gamma = \frac{(2q - n)r -2q}{q[mr - (m-1)]};$$
if, alternatively, 
$$\frac{m[(2q - n)r -2q]}{q[mr - (m-1)]} \geq \alpha_0,$$
then solutions are in $C^{0,\gamma}$, for any $0<\gamma <\frac{\alpha_{0}}{m} $.

We also obtain the $C^{0,\frac{\gamma}{\theta}}$ regularity in time, where 
$$\theta=2 - \left(1-\frac{1}{m}\right)\alpha=  \alpha \left( 1+\frac{1}{m} \right) + \left( 1-\alpha \right) 2 $$ 
is the $\alpha-$interpolation between $1+\frac{1}{m}$ and $2$. It is worth stressing that, as in the case of the $p$-Laplace equation, the integrability in time (respectively, in space) of the source affects the regularity in space (respectively, in time) of the solution. 

We remark that for $m=1$ we obtain
$$\gamma = 1 - \left( \frac{2}{r} + \frac{n}{q} -1 \right) \qquad {\rm and} \qquad \theta = 2,$$
recovering the optimal H\"older regularity for the non-homogeneous heat equation, in accordance with estimates obtained by energy considerations. 

For $n=1$,  it is proven in \cite{AroCaf86} that 
$$\alpha_0=\min \left\{1, \frac{1}{m-1}\right\} $$ 
but this is not the case in higher dimensions as corroborated by the celebrated counter-example in \cite{AroGra93}. The question of the sharp regularity for the homogeneous pme was recently addressed in \cite{GiaSil}, where it is shown that in the case $m\geq 2$  (see also \cite{Iva99} for $1<m<2$) a solution achieves the optimal modulus of continuity $C^{0, \frac{1}{m-1}}$ of the Barenblatt fundamental solution after a precise time lag, which is quantified in the paper. This optimal regularity issue is strongly intertwined with the regularity of the free boundary (cf. \cite{CafVazWol87}).

Observe that
$$\frac{m[(2q - n)r -2q]}{q[mr - (m-1)]} = \frac{2 m \left( 1- \displaystyle \frac{1}{r} - \frac{n}{2q}\right)}{\displaystyle m \left( 1-\frac{1}{r}\right) +  \frac{1}{r} } >0$$
and so indeed $\gamma >0$. Note also that 
$$\frac{m[(2q - n)r -2q]}{q[mr - (m-1)]} >1$$ 
if 
$$\left( 1 + \frac{1}{m}\right) \frac{1}{r}+\frac{n}{q} < 1,$$
and, as $q,r \rightarrow \infty$,
$$\frac{m[(2q - n)r -2q]}{q[mr - (m-1)]} \longrightarrow 2,$$
which means that after a certain integrability threshold it is the optimal regularity exponent of the homogeneous case that prevails, with
$$\alpha = \alpha_0^- \qquad {\rm and}  \qquad \gamma <\frac{\alpha_{0}}{m}<1.$$ 

\section{Weak solutions and approximation}

To fix ideas, we say a  locally bounded  function 
$$u \in C_{\rm loc} \left( 0,T ; L_{\rm loc}^{2}(U) \right), \qquad {\rm with} \quad |u|^{\frac{m+1}{2}} \in L_{\rm loc}^{2} \left( 0,T;W_{\rm loc}^{1,2}(U) \right)$$
is a local weak solution of \eqref{pme} if, for every compact set $K \subset U$ and every subinterval  $[t_{1}, t_{2}] \subset (0, T] $, we have
$$\left. \int_{K} u \varphi  \right|_{t_{1}}^{t_{2}} + \int_{t_{1}}^{t_{2}} \int_{K} \left\{ -u\varphi_{t} + m |u|^{m-1}\nabla u \cdot \nabla \varphi \right\} =\int_{t_{1}}^{t_{2}} \int_{K} f \varphi ,$$
for all  test functions  
$$\varphi \in W_{\rm loc}^{1,2} \left( 0,T;L^{2}(K) \right) \cap  L_{\rm loc}^{2} \left( 0,T;W_{0}^{1,2}(K) \right).$$
It is clear that all integrals in the above definition are convergent (cf. \cite[\S 3.5]{DiBeGiaVes11}), interpreting the gradient term as
$$|u|^{m-1}\nabla u : =  \frac{2}{m+1}\, {\rm sign} (u) \, |u|^{\frac{m-1}{2}} \nabla |u|^{\frac{m+1}{2}}.$$

An equivalent definition of weak solution involving the Steklov average is instrumental in obtaining the following Caccioppoli estimate (cf. \cite[\S 3.6]{DiBeGiaVes11}).

\begin{proposition} 
\label{caccio} 
Let $u$ be a local weak solution of \eqref{pme} and $K \times [t_{1}, t_{2}] \subset U \times (0, T]$. There exists a constant $C$, depending only on $n, m$ and  $K \times [t_{1}, t_{2}] $, such that 
$$\sup_{t_{1}< t < t_{2}} \int_{K}u^{2}\xi^{2} + \int_{t_{1}}^{t_{2}}\int_{K} |u|^{m-1}\vert\nabla u\vert^{2}\xi^{2} $$
$$ \leq  C\int_{t_{1}}^{t_{2}}\int_{K} u^{2}\xi \left|\xi_{t}\right|+\int_{t_{1}}^{t_{2}}\int_{K} |u|^{m+1} \left( \vert\nabla \xi\vert^{2} +\xi^2 \right)+ C\Vert f \Vert^{2}_{L^{q ,r}},$$
for all $\xi \in C_{0}^{\infty}(K \times (t_{1}, t_{2}))$ such that $\xi \in [0, 1].$ 
\end{proposition} 

\noindent The proof is standard and follows from testing the equation with $\varphi = u_h \xi^2$, where $u_h$ is the Steklov average of $u$, and performing the usual combination of integrating in time, passing to the limit in $h \rightarrow 0$ and applying Young's inequality. 

\medskip

We start our fine regularity analysis by fixing the intrinsic geometric setting for our problem. Given $0 < \alpha \leq 1$, let
\begin{equation}
\label{theta}
\theta : = 2 - \left(1-\frac{1}{m}\right)\alpha,
\end{equation}
which clearly satisfies the bounds
$$1+ \frac{1}{m} \leq \theta < 2 .$$ 
For such $\theta$, define the intrinsic $\theta $-parabolic cylinder as
$$G_{\rho} := \left( -\rho^\theta, 0 \right) \times B_{\rho}(0), \quad \rho > 0.$$

\medskip

We next use the available compactness to derive a mechanism linking solutions of the inhomogeneous pme and solutions of the homogeneous equation. This result is to be compared with a similar statement for the $p$-Laplace equation in \cite{TeiUrb14} (see also \cite{DuzMin05, BogDuzMin13}). 

\begin{lemma} \label{aprox}
Given $\delta >0$, there exists  $0 < \epsilon \ll 1$ such that if $\Vert f \Vert_{L^{q ,r}(G_{1})} \leq \epsilon$ and $u$ is a local weak solution of $(\ref{pme})$ in $G_{1}$, with $\Vert u \Vert_{\infty, G_{1}} \leq 1$, then there exists $\phi$ such that  
\begin{equation}
\phi_{t} - {\rm div} \left( m |\phi|^{m-1} \nabla \phi \right) = 0 \quad {\rm in}\ G_{1/2}
\label{homog}
\end{equation}
and
$$\Vert u -\phi \Vert_{\infty, G_{1/2}} \leq \delta .$$
\end{lemma}

\begin{proof}
Suppose, for the sake of contradiction, that, for some $\delta_{0} > 0$, there exist sequences $(u^{j})_{j}$ and $(f^{j})_{j}$, with 
$$u^{j} \in C_{\rm loc} \left( -1,0 ; L_{\rm loc}^{2}(B_1) \right), \quad \left| u^j \right|^{\frac{m+1}{2}} \in L_{\rm loc}^{2} \left( -1,0;W_{\rm loc}^{1,2}(B_1) \right)$$
and $f^{j} \in L^{q, r}(G_{1})$, such that
\begin{eqnarray}
\label{contrad. 1}
u_{t}^{j} - {\rm div} \left( m \left|u^{j}\right|^{m-1}\nabla u^{j} \right) = f^{j} \quad {\rm in} \ G_{1}\\
\label{contrad. 2} 
\Vert u^{j} \Vert_{\infty, G_{1}} \leq 1,\\
\label{contrad. 3} 
\Vert f^{j} \Vert_{L^{q ,r}(G_{1})} \  \leq 1/j,
\end{eqnarray}
but still, for any $j$ and any solution $\phi$ of the homogeneous equation in $G_{1/2}$,
\begin{eqnarray}
\Vert u^{j} -\phi \Vert_{\infty, G_{1/2}} > \delta_{0}.
\label{pqop}
\end{eqnarray}

Consider a cutoff function  $\xi \in  C_{0}^{\infty}(G_{1})$ such that $\xi \in [0,1], \xi \equiv 1$ in $G_{1/2}$ and $\xi \equiv 0$ near $\partial_{p}G_{1}$. 
Thus, since $u^{j}$ is a solution of \eqref{pme}, we can apply the Caccioppoli estimate of Proposition \ref{caccio} to get
$$\sup_{-1< t < 0} \int_{B_{1}}(u^j)^{2}\xi^{2} + \int_{-1}^{0}\int_{B_{1}} \left| u^{j}\right|^{m-1}\vert\nabla u^{j}\vert^{2}\xi^{2}$$
$$\leq C \int_{-1}^{0}\int_{B_{1}}(u^{j})^{2}\xi\vert \xi_{t}\vert + \int_{-1}^{0}\int_{B_{1}}\left| u^{j}\right|^{m+1} \left( \vert\nabla \xi\vert^{2}  + |\xi|^2 \right) + C \Vert f^{j} \Vert^{2}_{L^{q ,r}} \leq \tilde{c},$$
using \eqref{contrad. 2} and \eqref{contrad. 3}. Let us now define $v^{j} := \left| u^{j}\right|^{\frac{m+1}{2}}$. Observing that
$$ \vert \nabla v^{j} \vert^{2} = \left( \frac{m+1}{2} \right)^{2} \left| u^{j}\right|^{m-1} \vert \nabla u^{j} \vert^{2},$$
we obtain, 
$$\Vert \nabla v^{j} \Vert_{2, G_{1/2}}^{2} \leq \int_{-1}^{0}\int_{B_{1}}\vert\nabla v^{j} \vert^{2} \xi^{2} dxdt \leq \left( \frac{m+1}{2} \right)^{2}\tilde{c}$$
and then, for a subsequence, 
$$\nabla v^{j} \ \rightharpoonup \ \psi$$
weakly in $L^2 (G_{1/2})$. 

Moreover, owing to \cite[Thm. 16.1]{DiBeGiaVes11}, the equibounded sequence $(u^{j})_{j}$ is also equicontinuous and, by Arzel\`a--Ascoli theorem, 
$$u^{j} \ \longrightarrow \ \phi $$
uniformly in $G_{1/2}$, for yet another (relabelled) subsequence. Since we also have the pointwise convergence
$$ v^{j} := \left|u^{j}\right|^{\frac{m+1}{2}}\ \longrightarrow \ |\phi|^{\frac{m+1}{2}} =: v,$$
we can identify $\psi = \nabla v$.

Passing to the limit in \eqref{contrad. 1}, we find that $\phi$ solves \eqref{homog} which contradicts \eqref{pqop} for $j \gg 1$.
\end{proof}

\section{Sharp regularity via geometric iteration}
 
We now set up a geometric iteration, exploring the intrinsic scaling of the pme that will be crucial in obtaining the sharp H\"older exponent. The next result is the first step in this iteration. Let $\alpha_{0}$ denote the sharp H\"older continuity exponent for solutions of  \eqref{pme} in the homogeneous case and
$$\gamma = \frac{\alpha}{m},$$
with
\begin{equation}
0 < \alpha = \min\left\{\alpha_{0}^-, \frac{m[(2q - n)r -2q]}{q[mr - (m-1)]}\right\} < \alpha_{0} \leq \min\left\{1,\frac{1}{m-1}\right\}\leq1.
\label{alfa}
\end{equation}

\begin{lemma}\label{Lemma 3.3}
There exists  $\epsilon >0, $ and $0 < \lambda \ll 1/2$, depending only on $m,n$ and $\alpha$, such that if $\Vert f \Vert_{L^{q ,r}(G_{1})} \leq \epsilon$ and $u$ is a local weak solution of \eqref{pme} in $G_{1}$, with $\Vert u \Vert_{\infty, G_{1}} \leq 1$,  then 
$$\Vert u \Vert_{\infty, G_{\lambda}} \leq \lambda^{\gamma}$$
provided 
$$|u(0,0)|\leq \frac{1}{4} \lambda^{\gamma}.$$
\end{lemma}

\begin{proof}
Take $0<\delta <1$, to be chosen later, and apply Lemma \ref{aprox} to obtain $0 < \epsilon \ll 1$ and a solution $\phi$ of \eqref{homog} in $G_{1/2}$  such that
$$\Vert u -\phi \Vert_{\infty, G_{1/2}} \leq \delta.$$
Since $\phi$ solves \eqref{homog}, it follows from the available regularity theory (cf.  \cite[Thm. 16.1]{DiBeGiaVes11}) that $\phi$ is locally $C_{x}^{\alpha_{0}} \cap C_{t}^{\alpha_{0}/2} $,  for $ 0 < \alpha_{0} < 1$. Thus we obtain 
\begin{eqnarray*}
\underset{(x, t) \in G_{\lambda}}{\sup} \vert \phi(x, t) - \phi(0, 0) \vert \leq C\lambda^{\frac{\alpha_{0}}{m}},
\end{eqnarray*}
for $\lambda \ll 1,$ to be chosen soon, and $C > 1$ universal. In fact, for $(x, t) \in G_{\lambda}$,
\begin{eqnarray*}
\vert \phi(x, t) - \phi(0, 0) \vert &\leq &  \vert \phi(x, t) - \phi(0, t) \vert + \vert \phi(0, t) - \phi(0, 0) \vert \\
&\leq & c_{1}\vert x - 0 \vert^{\alpha_{0}} +  c_{2}\vert t - 0 \vert^{\alpha_{0}/2}\\
&\leq & c_{1}\lambda^{\alpha_{0}} + c_{2}\lambda^{\frac{\theta}{2}\alpha_{0}}\\
&\leq & C\lambda^{\frac{\alpha_{0}}{m}}
\end{eqnarray*}
since $\theta \geq 1+ \frac{1}{m} >\frac{2}{m}$. We can therefore estimate 
\begin{eqnarray}
\sup_{G_\lambda} |u| &\leq &  \sup_{G_{1/2}}  |u-\phi | + \sup_{G_\lambda} |\phi-\phi(0,0)| + |\phi(0,0)-u(0,0)| + |u(0,0)| \nonumber \\
&\leq & 2 \delta + C\lambda^{\frac{\alpha_{0}}{m}}+ \frac{1}{4} \lambda^{\gamma}. \label{solskajer}
\end{eqnarray}
Note that we will choose $\lambda \ll 1/2$ and thus 
\begin{eqnarray*}
G_{\lambda} := (-\lambda^{\theta}, 0) \times B_{\lambda} \subset (-(1/2)^{\theta}, 0) \times B_{1/2} = G_{1/2}.
\end{eqnarray*}

We finally fix the constants, choosing 
$$\lambda = {\left(\frac{1}{4C}\right)^{\frac{m}{\alpha_{0} - \alpha}}} \qquad {\rm and} \qquad \delta= \frac{1}{4} \lambda^{\gamma} ,$$
and fixing also $\epsilon >0$, through Lemma \ref{aprox}. The result follows from estimate $(\ref{solskajer})$ with the indicated choices. 
\end{proof}

We now iterate the previous result in the appropriate geometric setting. 

\begin{theorem}
There exists  $\epsilon >0, $ and $0 < \lambda \ll 1/2$, depending only on $m,n$ and $\alpha$, such that if $\Vert f \Vert_{L^{q ,r}(G_{1})} \leq \epsilon$ and $u$ is a local weak solution of \eqref{pme} in $G_{1}$, with $\Vert u \Vert_{\infty, G_{1}} \leq 1,$  then 
\begin{equation}
\Vert u  \Vert_{\infty, G_{\lambda^{k}}} \leq  (\lambda^k)^{\gamma}
\label{xuxu}
\end{equation}
provided
$$|u(0,0)|\leq \frac{1}{4} \left( \lambda^k \right)^{\gamma}.$$
\label{pois}
\end{theorem}

\begin{proof}
The proof is by induction on $ k \in \mathbb{N}$. If $k = 1$, \eqref{xuxu} holds due to Lemma \ref{Lemma 3.3}. Now suppose the conclusion holds for $k$ and let's show it also holds for $k+1$. Consider the function $v:G_{1} \rightarrow \R$  defined by
$$v(x,t) = \frac{u(\lambda^{k}x, \lambda^{k\theta}t)}{\lambda^{\gamma k}}.$$
We have 
$$v_{t}(x,t) = \lambda^{k\theta - \gamma k}u_{t}(\lambda^{k}x, \lambda^{k\theta}t),$$

$$\nabla v(x,t) = \lambda^{k - \gamma k}\nabla u(\lambda^{k}x, \lambda^{k\theta}t)$$
and
$${\rm div} \left( m \left| v(x,t) \right|^{m-1}\nabla v(x,t) \right) $$
$$= \lambda^{k(2-\alpha)} {\rm div} \left( m \left| u(\lambda^{k}x, \lambda^{k\theta}t) \right|^{m-1}\nabla u(\lambda^{k}x, \lambda^{k\theta}t) \right).$$
Recalling $(\ref{theta})$, we conclude, since $u$ is a local weak solution of \eqref{pme} in $G_{1}$, that
$$v_{t}- {\rm div} \left( m |v|^{m-1}\nabla v \right)= \lambda^{k(2-\alpha)}f(\lambda^{k}x, \lambda^{k\theta}t)= \tilde{f}(x,t).$$
We now compute 
\begin{eqnarray*}
\Vert \tilde{f} \Vert_{L^{q,r}(G_{1})}^{r} &=& \int_{-1}^{0} \left(\int_{B_{1}} \left\vert \tilde{f}(x,t) \right\vert^{q} dx \right)^{r/q}dt\\
&=& \int_{-1}^{0}\left(  \int_{B_{1}} \lambda^{k(2 - \alpha )q}     \left\vert  f(\lambda^{k}x, \lambda^{k\theta}t) \right\vert^{q} dx \right)^{r/q}dt\\
&=& \int_{-1}^{0}\left(  \int_{B_{\lambda^k}} \lambda^{k(2 - \alpha)q - kn} \left\vert  f(x, \lambda^{k\theta}t) \right\vert^{q} dx \right)^{r/q}dt\\
&=&  \lambda^{[k(2 - \alpha )q - kn]\frac{r}{q} } \int_{-1}^{0}\left( \int_{B_{\lambda^{k}}}  \left\vert f(x, \lambda^{k\theta}t) \right\vert^{q} dx \right)^{r/q}dt\\
&=&  \lambda^{[k(2 - \alpha )q - kn]\frac{r}{q} - k\theta}  \int_{-\lambda^{k\theta}}^{0}\left(  \int_{B_{\lambda^{k}}}  \left\vert f(x, t) \right\vert^{q} dx \right)^{r/q}dt.\\
\end{eqnarray*}

Because of the crucial and optimal choice of $\alpha$ in \eqref{alfa}, we have
$$\left[ k(2 - \alpha)q - kn \right]\frac{r}{q} - k\theta \geq  0$$
and thus
\begin{eqnarray*}
\Vert \tilde{f} \Vert_{L^{q,r}(G_{1})} \leq \Vert f \Vert_{L^{q,r}( (-\lambda^{\theta k}, 0) \times B_{\lambda^{k}})} \leq \Vert f \Vert_{L^{q,r}(G_{1})} \leq \epsilon,
\end{eqnarray*}
which entitles $v$ to Lemma \ref{Lemma 3.3}. Note that $\Vert v \Vert_{\infty,G_{1}} \leq 1$, due to the induction hypothesis, and 
$$\left| v(0,0) \right| = \left| \frac{u(0,0)}{ \left( \lambda^k \right)^{\gamma}} \right| \leq \left| \frac{\frac{1}{4} \left( \lambda^{k+1} \right)^{\gamma}}{ \left( \lambda^k \right)^{\gamma}} \right| \leq \frac{1}{4} \lambda^{\gamma}.$$ 
It then follows that 
$$\Vert v \Vert_{\infty, G_{\lambda}} \leq \lambda^{\gamma},$$
which is the same as
$$\Vert u \Vert_{\infty, G_{\lambda^{k+1}}} \leq  \lambda^{\gamma(k+1)}.$$
The induction is complete. 
\end{proof}

We next show the smallness regime required in the previous theorem is not restrictive and generalize it to cover the case of any small radius.

\begin{theorem}
If $u$ is a local weak solution of \eqref{pme} in $G_{1}$ then, for every $0<r<\lambda$, we have
$$\|u  \|_{\infty, G_{r}} \leq  C\, r^{\gamma}$$
provided
$$|u(0,0)|\leq \frac{1}{4} r^{\gamma}.$$
\label{sim}
\end{theorem}

\begin{proof}
Take
$$v(x,t) = \rho u \left( \rho^{a}x, \rho^{(m-1) + 2a}t \right) $$
with $\rho, a$ to be fixed, which solves
$$v_{t}- {\rm div} (m |v|^{m-1}\nabla v)= \rho^{m + 2a}f(\rho^{a}x, \rho^{(m-1)+ 2a)}t) = \tilde{f}(x,t).$$
We have
$$\Vert v \Vert_{\infty,G_{1}} \leq  \rho \Vert u \Vert_{\infty,G_{1}}$$
and
$$\Vert \tilde{f} \Vert_{L^{q,r}(G_{1})}^{r} = \rho^{(m + 2a)r - a(n\frac{r}{q} + 2) - (m-1)}\Vert f \Vert_{L^{q,r}(G_{1})}^{r}.$$
Choosing $a > 0$ such that
$$(m + 2a)r - a \left( \frac{nr}{q} + 2 \right) - (m-1) > 0,$$
which is always possible, and $0 < \rho \ll 1$, we enter the smallness regime required by Theorem \ref{pois}, i.e., 
$$\Vert v \Vert_{\infty,G_{1}} \leq 1  \qquad {\rm and} \qquad \Vert \tilde{f} \Vert_{L^{q,r}(G_{1})} \leq \epsilon.$$

\medskip

Now, given $0<r<\lambda$, there exists $k \in \N$ such that 
$$\lambda^{k+1} < r \leq \lambda^{k}.$$
Since $|u(0,0)|\leq \frac{1}{4} r^{\gamma} \leq  \frac{1}{4} (\lambda^{k})^{\gamma}$, it follows from Theorem \ref{pois} that
$$\Vert u  \Vert_{\infty, G_{\lambda^{k}}} \leq  (\lambda^k)^{\gamma}.$$
Then, for $C=\lambda^{-\gamma}$,
$$\Vert u  \Vert_{\infty, G_{r}} \leq \Vert u  \Vert_{\infty, G_{\lambda^{k}}} \leq  (\lambda^k)^{\gamma} < \left( \frac{r}{\lambda} \right)^{\gamma} = C\, r^{\gamma} .$$
\end{proof}

We now complete our study, with the main result of the paper.

\begin{theorem}
Let $u$ be a locally bounded weak solution of \eqref{pme} in $G_1$, with $f \in L^{q,r}$ satisfying $(\ref{borderline1})$.
Then $u$ is locally of class $C^{0,\gamma}$ in space and $C^{0,\frac{\gamma}{\theta}}$ in time, with
$$\gamma =\frac{\alpha}{m}  , \qquad \alpha =  \min\left\{\alpha_{0}^-, \frac{m[(2q - n)r -2q]}{q[mr - (m-1)]}\right\} .$$
Here $0< \alpha_{0} \leq 1$ denotes the optimal H\"older exponent for solutions of the homogeneous case and $\theta$ is given in \eqref{theta}. 
\end{theorem}

\begin{proof}
We study the H\"older continuity at the origin, proving there is a uniform constant $K$ such that 
\begin{equation}
\|u-u(0,0)\|_{\infty, G_r} \leq K r^\gamma.
\label{connor}
\end{equation}
We know, \textit{a priori}, that $u$ is continuous so we can define 
$$\mu:= (4| u(0,0) |)^{1/\gamma} \geq 0.$$
Take any radius $0<r<\lambda$. We analyse three alternative cases, exhausting all possibilities.

\medskip

\begin{itemize}

\item If $\mu \leq r<\lambda$ then, by Theorem \ref{sim},
\begin{equation}
\sup_{G_r} \left| u(x,t) - u(0,0)\right| \leq C\, r^{\gamma} + |u(0,0)| \leq \left( C+\frac{1}{4}\right) r^{\gamma}.
\label{uno}
\end{equation}

\medskip

\item If $0<r<\mu$, we consider the function 
$$w(x,t):=\frac{u(\mu x, \mu^\theta t)}{\mu^{\gamma}}.$$
Clearly, $|w(0,0)|=\frac14$ and $w$ solves in $G_1$ the pme
$$w_{t} - {\rm div} \left( m |w|^{m-1} \nabla w \right) = \mu^{2-\alpha} f(\mu x, \mu^\theta t).$$
Moreover, again using Theorem \ref{sim}, it follows that 
$$\| w  \|_{\infty, G_{1}}  = \mu^{-\gamma} \| u  \|_{\infty, G_{\mu}}  \leq C,$$
since $| u(0,0) | = \frac{1}{4} \mu^{\gamma}$. With this uniform estimate in hand, and using local $C^{0,\alpha}$ regularity estimates \cite{And91}, we find that there exists a radius $\rho_0$, depending only on the data, such that
$$| w (x,t)| \geq \frac{1}{8}, \quad \forall\, (x,t) \in G_{\rho_0}.$$
This implies that, in $G_{\rho_0}$, $w$ solves a uniformly parabolic equation of the form
$$w_{t} - {\rm div} \left( a(x,t) \nabla w \right) = f \in L^{q,r},$$
with continuous coefficients satisfying the bounds $0<c_1\leq a(x,t) \leq c_2$. In particular, we have (see \cite{TeiUrb14})
$$w \in C^{0,\beta} (G_{\rho_0}), \quad \mbox{with }  \ \beta = 1 - \left( \frac{2}{r} + \frac{n}{q} -1 \right) > \gamma,$$
which is the optimal H\"older regularity for solutions of the heat equation with a source in $L^{q,r}$, for exponents satisfying \eqref{borderline1}. 
As an immediate consequence,
$$\sup_{(x,t) \in G_r} \left| w(x,t) -  w(0,0) \right| \leq C\, r^{\beta}, \quad \forall \, 0<r<\frac{\rho_0}{2},$$
which, in terms of $u$, reads
$$\sup_{(x,t) \in G_r} \left| \frac{u(\mu x, \mu^\theta t)}{\mu^{\gamma}} - \frac{u(0,0)}{\mu^{\gamma}}\right| \leq C\, r^{\beta},  \quad \forall \, 0<r<\frac{\rho_0}{2}.$$
Since $\gamma < \beta$,  we conclude 
$$\sup_{(x,t) \in G_{\mu r}} \left| u(x,t)  - u(0,0)  \right| \leq C\, (\mu r)^{\gamma}, \quad \forall \, 0<\mu r<\mu\frac{\rho_0}{2},$$
and, relabelling, we obtain
\begin{equation}
\sup_{(x,t) \in G_{r}} \left| u(x,t)  - u(0,0)  \right| \leq C\, r^{\gamma}, \quad \forall \, 0<r<\mu\frac{\rho_0}{2}.
\label{dos}
\end{equation}

\medskip

\item Finally, for $\mu\frac{\rho_0}{2} \leq r < \mu$, we have
\begin{eqnarray}
\sup_{(x,t) \in G_{r}} \left| u(x,t)  - u(0,0)  \right| & \leq &   \sup_{(x,t) \in G_{\mu}} \left| u(x,t)  - u(0,0)  \right| \nonumber\\
& \leq & C\, \mu^{\gamma} \leq C \left( \frac{2r}{\rho_0} \right)^{\gamma}= \tilde{C} r^{\gamma} \label{tres}.
\end{eqnarray}

\end{itemize}

Putting $K=\max \left\{ C+\frac14,  \tilde{C} \right\}$ and combining \eqref{uno}--\eqref{tres}, we obtain \eqref{connor}, for every $0<r<\lambda$, and the proof is complete.
\end{proof}

\medskip

\noindent{\bf Acknowledgments.} D.J.A. supported by CNPq - Brazil. A.F.M. supported by CAPES - Brazil. J.M.U. partially supported by the Centre for Mathematics of the University of Coimbra -- UID/MAT/00324/2013, funded by the Portuguese government through FCT/MCTES and co-funded by the European Regional Development Fund through Partnership Agree\-ment PT2020.

\medskip

\bibliographystyle{amsplain, amsalpha}

\begin{thebibliography}{99}

\bibitem{And91} 
D. Andreucci,
{\it $L_{\rm loc}^\infty$-estimates for local solutions of degenerate parabolic equations},
SIAM J. Math. Anal. 22 (1991), 138--145. 

\bibitem{AraTeiUrb18} 
D.J. Ara\'ujo, E. V. Teixeira and J.M. Urbano, 
{\it A proof of the $C^{p'}$-regularity conjecture in the plane},  
Adv. Math. 316 (2017), 541--553.

\bibitem{AraTeiUrb17} 
D.J. Ara\'ujo, E. V. Teixeira and J.M. Urbano, 
{\it Towards the $C^{p'}$-regularity conjecture in higher dimensions},  
Int. Math. Res. Not. IMRN 2018 (2018), 6481--6495.

\bibitem{AraZhang17} 
D.J. Ara\'ujo and L. Zhang, 
{\it Interior $C^{1, \alpha}$ estimates for $p$-Laplacian equations with optimal regularity},  
Commun. Contemp. Math., to appear.

\bibitem{AroCaf86} 
D. G. Aronson and L.A. Caffarelli, 
{\it Optimal regularity for one-dimensional porous medium flow},
Rev. Mat. Iberoamericana 2 (1986), 357--366.
 
\bibitem{AroGra93} 
D. G. Aronson and J. Graveleau, 
{\it A self-similar solution to the focusing problem for the porous medium equation},
European J. Appl. Math. 4 (1993), 65--81.

\bibitem{BogDuzMin13} 
V. B\"ogelein, F. Duzaar and G. Mingione,
{\it The regularity of general parabolic systems with degenerate diffusion},
Mem. Amer. Math. Soc. 221 (2013), no. 1041, vi+143 pp.

\bibitem{CafVazWol87} 
L.A. Caffarelli, J.L. V\'azquez and N.I. Wolanski,
{\it Lipschitz continuity of solutions and interfaces of the N-dimensional porous medium equation},
Indiana Univ. Math. J. 36 (1987), 373--401. 

\bibitem{DiBe93}
E. DiBenedetto,
{\it Degenerate Parabolic Equations}, 
Universitext, Springer-Verlag, New York, 1993.

\bibitem{DiBeGiaVes11} 
E. DiBenedetto, U. Gianazza and V. Vespri,
{\it Harnack's Inequality for Degenerate and Singular Parabolic Equations},
Springer Monographs in Mathematics, Springer, New York, 2012.
 
\bibitem{DuzMin05} 
F. Duzaar and G. Mingione,
{\it Second order parabolic systems, optimal regularity, and singular sets of solutions},
Ann. Inst. H. Poincar\'e Anal. Non Lin\'eaire 22 (2005), 705--751.

\bibitem{GiaSil} 
U. Gianazza and J. Siljander,
{\it Sharp regularity for weak solutions to the porous medium equation},
 arXiv:1607.06924v1 [math.AP].

\bibitem{Iva99} 
A. V. Ivanov, 
{\it Gradient estimates for doubly nonlinear parabolic equations},
Zap. Nauchn. Sem. S.-Peterburg. Otdel. Mat. Inst. Steklov. (POMI) 233 (1996), Kraev. Zadachi Mat. Fiz. i Smezh. Vopr. Teor. Funkts. 27, 63--100, 256; translation in J. Math. Sci. (New York) 93 (1999), 661--688.

\bibitem{TeiUrb14} 
E. V. Teixeira and J.M. Urbano, 
{\it A geometric tangential approach to sharp regularity for degenerate evolution equations},  
Anal. PDE 7 (2014), 733--744.

\bibitem{Urb08}
J.M. Urbano,
{\it The Method of Intrinsic Scaling},
Lecture Notes in Mathematics 1930, Springer-Verlag, Berlin, 2008.

\bibitem{Vaz07} 
J.L. V\'azquez, 
{\it The Porous Medium Equation. Mathematical Theory},
Oxford Mathematical Monographs, Oxford University Press, Oxford, 2007.

\end{thebibliography}

\end{document}